\documentclass[preprint,12pt,authoryear]{elsarticle}

\newcommand{\comment}[1]{}

\usepackage{url}
\usepackage{hyperref}
\usepackage{stmaryrd}
\usepackage{appendix}
\usepackage{latexsym}
\usepackage{amsfonts,amsmath,amsthm,mathrsfs,amssymb,graphicx}
\usepackage[mathscr]{euscript}
\usepackage{verbatim}
\usepackage{subfigure}
\usepackage{wrapfig}

\newtheorem{theorem}{Theorem}[section]

\newtheorem{definition}{Definition}[section]

\newtheorem{lemma}{Lemma}[section]

\numberwithin{equation}{section}

\def\R{\hbox{\bf R}}

\def\N{\hbox{\bf N}}

\def\B {\hbox{\bf B}}

\usepackage{amssymb}

\begin{document}

\begin{frontmatter}



\title{Computational Topology: Isotopic Convergence to a Stick Knot}


\author[label1]{J. Li}
\address[label1]{University of Connecticut, Storrs, CT 06269, U.S.A.}

\author[label1]{T. J. Peters\footnote{This author (and the 1st author) was partially supported by NSF grant CNS 0923158 and by IBM Faculty Award 2013 -- 6327140.  All statements here are the responsibility of the author, not of the National Science Foundation nor of IBM.} }

\author[label2]{K. E. Jordan}
\address[label2]{IBM T.J. Watson Research Center, Cambridge, MA, U.S.A.}

\author[label3]{P.  Zaffetti}
\address[label3]{iDevices Inc., Avon, CT 06001, U.S.A.}

\begin{abstract}
Computational topology is a vibrant contemporary subfield and this article integrates knot theory and mathematical visualization.  Previous work on computer graphics developed a sequence of smooth knots that were shown to converge point wise to a piecewise linear (PL) approximant.  This is extended to isotopic convergence, with that discovery aided by computational experiments.  Sufficient conditions to attain isotopic equivalence can be determined {\it a priori}.  These sufficient conditions need not be tight bounds, providing opportunities for further optimizations.  The results presented will facilitate further computational experiments on the theory of PL knots (also known as stick knots), where this theory is less mature than for smooth knots.   

\end{abstract}


\end{frontmatter}





\section{Introduction \& Related Work}


For a positive integer $n$,  a B\'ezier curve \cite{Piegl}of degree $n$  is defined as $\B(t)$,  with control points $P_i \in\mathbb{R}^3$ by
$$\B(t) = \sum_{i=0}^{n} \binom{n}{i} t^i (1-t)^{n-i} P_i, \hspace{2ex} t \in [0,1].$$
The curve $\mathcal{P}$ formed by PL interpolation on the ordered set of points $\{P_0,P_1,\ldots,P_n\}$ is called the {\em control polygon}.  This $\mathcal{P}$  is a PL approximation of $\B$.

The curves considered here will be closed by understanding that $P_0 = P_n$.  Furthermore, it is assumed that both the B\'ezier curves and their control polygons are simple.  The focus here is on the isotopic equivalence between a knotted B\'ezier curve and its PL approximation.   These knotted PL approximations are also known as `stick knots' \cite{Colin}.

\pagebreak

A contemporary treatment of knots and molecules  \cite{WeMi06} provided motivation for much of this work.  In particular, remarks on the ``... humbling ...'' status in theoretical understanding of stick knots versus smooth knots stimulated the consideration, here, of B\'ezier curves that are isotopically equivalent to stick knots.
 
The term `molecular movies' was introduced \cite{MMovies} to include the visualization of molecular simulations.  There is a specific cautionary example  \cite{LiPeMaKoJo15} about introducing topological artifacts into a  molecular movie.  The relevant knots were $4_1$ and the unknot, represented in sufficiently simple format that isotopic equivalences were easily determined by standard methods \cite{Livingston1993}.  Knot visualization software \cite{TJPweb} was used to develop that illustrative example.

Those application specific considerations \cite{LiPeMaKoJo15} led to the generalizations presented here.  The theorem presented here generalizes that example to sufficient conditions for convergence between B\'ezier curves and their isotopic PL approximations.  An \emph{a priori} bound is given on the number of iterations needed to obtain an isotopic approximation.

The preservation of topological characteristics in computational applications is of contemporary interest \cite{Amenta2003, L.-E.Andersson2000, Andersson1998, Chazal2005,Chazal2011, DenneSullivan2008, ENSTPPKS12, bez-iso,TJP08,KiSi08, JL-isoconvthm,Maekawa_Patrikalakis_Sakkalis_Yu1998}.  The isotopy theorems on knots of finite total curvature are foundational to the work presented here.  Sufficient conditions for a homeomorphism between a B\'ezier curve and its control polygon have been studied \cite{M.Neagu_E.Calcoen_B.Lacolle2000}, while topological differences have also been shown \cite{Bisceglio, JL2012, Piegl}.   Sufficient conditions were given to insure that perturbations of the control points maintain isotopic equivalence of the perturbed splines  \cite{andersson2000equivalence}.   There is an example of a PL structure that becomes self-intersecting while the associated B\'ezier curve remains simple \cite{CPIJ2014}.  Recent perspectives on computational topology have appeared
\cite{GC09, EdHa10, Zo05}.

\section{Inserting Midpoints as Control Points}
\label{sec:dr}

The fundamental approximation technique introduced here is dual to many others.  The more typical focus is to produce a sequence of PL curves that approximates a given smooth curve.  Indeed, these authors have previously published such results \cite{JoLiPeRo14}.    The duality here is to create a sequence of smooth knots that converge to a PL knot, achieving isotopic equivalence within the sequence.   The use of smoothness here is understood to be $C^{\infty}$, with a possible exception at the point $\B(0) = \B(1)$.   This technique has been called \emph{collinear insertion} \cite{LiPeMaKoJo15}, where an example was presented with 8 initial control points (inclusive of equality of the initial and final control points).  That control polygon was the knot $4_1$ and its associated B\'ezier curve was the unknot.  With 4 iterations of collinear insertion, the stick and smooth knots were isotopic.

\pagebreak

\begin{definition}
Consider the control polygon $\mathcal{P}$.  To avoid trivial cases, it is assumed that, for $i > 0$,
\begin{itemize}
\vspace{-0.09in}
\item $P_i \neq P_i*$ for any $i \neq i*$, 
\vspace{-0.09in}
\item only two such $P_i \neq P_i*$ can be collinear and 
\vspace{-0.09in}
\item $n \geq 4$.
\vspace{-0.09in}
\end{itemize}
\end{definition}

Sequences of control polygons and B\'ezier curves will be generated by letting $\mathcal{P}^{(0)} = \mathcal{P}$ and generating $\mathcal{P}^{(1)}$ from $\mathcal{P}^{(0)}$ by adding the midpoint of each edge of $\mathcal{P}^{(0)}$.  For $j \geq 0$, similarly generate $\mathcal{P}^{(j+1)}$ by the insertion of midpoints between all of the control points of $\mathcal{P}^{(j)}$.   For each $j$, the corresponding B\'ezier curve will be denoted as $\B^{(j)}$.  Note that all $\mathcal{P}^{(j)}$'s are isotopic under the trivial identification map.

\section{Convergence Theorem}
\label{sec:convthm}

The primary convergence result relies upon a previously published theorem \cite[Theorem 4.2]{DenneSullivan2008} on rectifiable \cite{rect} graphs of finite total curvature \cite{Milnor1950}.   This central theorem is quoted, below, after definition of a key notion of closeness.

\begin{definition}
Given two rectifiable embeddings $\Gamma_1$  and $\Gamma_2$ of the same combinatorial graph,
we say they are $(\delta, \theta)-close$ if there exists a homeomorphism between them such that corresponding points are within distance $\delta$ of each other, and corresponding tangent vectors
are within angle $\theta$ of each other almost everywhere.
\end{definition}

\begin{theorem}
\cite[Theorem 4.2] {DenneSullivan2008} Suppose $\Gamma_1$ is a knotted graph of finite total curvature and 
$\epsilon > 0$  is given. Then there exists $\delta > 0$ such that any (rectifiable) graph $\Gamma_2$ which is $(\delta, \pi/8)-close$ to $\Gamma_1$ is ambient isotopic to $\Gamma_1$, via an isotopy which moves no point by more than $\epsilon$.
\label{thm:delta-theta-iso}
\end{theorem}

The tangent vectors of the B\'ezier curves and their control polygons are well defined except at possibly finitely many points.  The homeomorphism between the knotted control polygon and knotted B\'ezier curve will be the natural one matching points defined at the same parametric values over $[0, 2\pi].$   The PL knots considered will have finite total curvature.   The result that every compact connected set which has finite length is rectifiable \cite{david1993analysis} is directly applicable to the curves considered.

There is pervasive interest in computer graphics in convergence properties between B\'ezier curves and their PL approximations \cite{Piegl}, motivating the general convergence considered here for collinear insertion.  There are supportive techniques already available for B\'ezier curves, as summarized briefly here.  

The derivative of $\B$ is a B\'ezier curve \cite{Farin}, with control points $n \ast \Delta P_i $, in terms of  the forward difference operator, denoted by  \[ \Delta P_i = P_{i  + 1}- P_i, \hspace{1ex} i = 0, \ldots, n - 1.\]  The second iterated forward difference operator is defined similarly as  
\[  \Delta_2 P_i = P_{i+2} - 2P_{i + 1} +  P_i.\]

Point wise convergence in distance under $j$ collinear insertions has been shown \cite{LiPeMaKoJo15}. Bounds will be shown on the rate of convergence for collinear insertion by adapting a previously published inequality \cite{Nairn-Peters-Lutterkort1999}.  Each control point of $\B$ is an ordered triple $(x,y,z) \in \R^3.$.  Let $\Delta_2 P^x$ denote the second iterated forward difference taken over the $x$ coordinate (with similarly definitions regarding the $y$ and $z$ coordinates).  Then, $\Delta_2 P^x$ is a vector having $n - 2$ entries and the usual 1-norm is indicated by $\| \Delta_2 P^x \|_1.$  Let $$\| \Delta_2 P \|_{1,M} = \max_{w \in \{x,y,z\}} \| \Delta_2 P^w \|_1.$$  

The notation \cite{Nairn-Peters-Lutterkort1999} of $ N_1(n)$ and a related combinatorial upper bound appears in the appendix of  Section~\ref{sec:suppcombo}. The indicated adaptation was previously implicit \cite{LiPeMaKoJo15}but is explicated as
\begin{equation}
\label{ineq:polydist} 
 \| \B(t) - \mathcal{P}(t)\|_{\infty, [0,1]} \leq N_1(n) \ast \| \Delta_{2}P \|_{1, M},
\end{equation}
leading directly to 
\begin{equation}
\label{ineq:drate}
 \| \B^{(j)}(t) - \mathcal{P}^{(j)}(t)\|_{\infty, [0,1]}  \leq  N_1(2^j n) \ast \| \Delta_2P^{(j)}\|_{1,M} \leq (n/(4 \sqrt{n \ast 2^j +1})) \| \Delta_{2}P \|_{1,M}.
\end{equation}

\subsection{Distance Bounds: Curve and Hodograph}
\label{ssec:distclose}
An upper bound on the distance between a hodograph and its control polygon is also determined directly from Inequality~\ref{ineq:polydist}.
Let ${\mathbf {\frac{d}{dt}}} \B$ denote the derivative of $\B$ and ${\mathbf {\frac{d}{dt}}} \B$ is a B\'ezier curve with control points of the form
\begin{equation}
 n \ast  (P_{i} - P_{i-1}),\hspace{2ex} i = 1, \ldots, n - 1.
\label{eq:delta2hod}
\end{equation}

Let $H_1, H_2, \ldots, H_{n -1}$ be the control points of ${\mathbf {\frac{d}{dt}}} \B$ and let $\mathcal{H}$ be the associated control polygon, leading to the inequality
\begin{equation}
\label{ineq:hoddist} 
 \| {\mathbf {\frac{d}{dt}}} \B(t) - \mathcal{H}(t)\|_{\infty, [0,1]} \leq N_1(n) \ast \| \Delta_{2}H \|_{1,M}.
\end{equation}
A slight refinement has been made on the right hand side of Inequality~\ref{ineq:hoddist} in the use of $N_1(n)$ in place of $N_1(n - 1)$, as $N_1$ is a `` ... slowly growing function ...'' \cite{Nairn-Peters-Lutterkort1999}.  The advantage of this minor replacement is to continue to evaluate $N_1$ on $n$ as an even integer, where that convenience is justified in Section~\ref{sec:suppcombo}.  While this increases the upper bound, the increase is not substantive for theoretical convergence questions, though a tighter upper bound may be valuable to consider on specific data.

For each $j \in \N \cup \{0\}$, let $\mathcal{H}^{(j)}$ denote the control polygon for $\mathbf {\frac{d}{dt}} \B^{(j)}(t).$ By definition, 
\[\Delta_2 H_{i + 2} = n \ast  ((P_{i + 2} - P_{i+1}) - 2 (P_{i + 1} - P_{i}) + (P_{i + 2} - P_{i-1})) ,\] which simplifies to 
\begin{equation}
\label{eq:4pts}
\Delta_2 H_{i +2}= n \ast (P_{i + 2} - 3 P_{i+1} + 3 P_{i} - P_{i-1}).
\end{equation}

\begin{lemma}
\label{lem:collapse-delta2}
For each $j \in \N \cup \{0\}, \hspace{1ex} \| \Delta_2  H^{(j)} \|_{1,M} \leq  2^{j - 1} \| \Delta_2 H \|_{1,M}. $
\end{lemma}

\begin{proof}
For each $j$, there are two important cases to consider for $\Delta_2 H^{(j)}$. 
Recall that the distance between successive control points of $ P^{(j)}$,  is equal $2^{-j} \| P_{\ell} - P_{\ell-1}\|$, for $\ell = 1, \ldots n $. 

\emph{Case 1}: The four points used in Equation~\ref{eq:4pts} are collinear and $\Delta_2 H_i^{(j)}$ is zero over all coordinates $x, y,z$.

\emph{Case 2}: Three of the four points in Equation~\ref{eq:4pts} are collinear.  If the first three points are collinear, then $\| \Delta_2 H_i^{(j)}\| = 2^{-j} \ast n  \|P_{\ell} - P_{\ell - 1}\|$  for some $\ell = 1, \ldots n.$   If the last three points are collinear, then the same value is obtained.  Hence, in evaluation of the 1-norm for the vector $\Delta_2 \mathcal{H}^{(j)}$, this term occurs twice for each $\ell$.    

The conclusion from these two cases is that 
\[ \| \Delta_2 \mathcal{H}^{(j)} \| = 2^{-(j - 1)} \| \Delta_2 \mathcal{H} \|. \hspace*{5ex} \Box  \]
\end{proof}

Applying  Lemma~\ref{lem:collapse-delta2} and Inequality~\ref{ineq:drate} to the hodograph yields

\begin{equation}
\label{ineq:hod-rate}
\|\mathbf {\frac{d}{dt}}  \B^{(j)}(t) - \mathcal{H}^{(j)}(t)\|_{\infty, [0,1]}  \leq (n/(2 \sqrt{n \ast 2^j + 1})) \ast 2^{-(j-1)} \ast \| \Delta_{2}H \|_{1,M}.
\end{equation}

\subsection{Bounding Angles Between Tangent Vectors}
\label{ssec:angtangents}

Inequality~\ref{ineq:hod-rate} provides a bound on distance between each point on the hodograph, ${\frac{d}{dt}}  \B^{(j)}(t)$ and the corresponding point
$\mathcal{H}^{(j)}(t)$ on the control polygon, which suffices as a bound between corresponding tangent vectors.

For each non-negative integer $j$, there are $2^j n + 1$ control points of $\mathcal{P}^{(j)}$.  Let $k = 0, \ldots, 2^j n$ and define a uniform parametrization of $[0,1]$ with $\mathcal{P}^{(j)}(k/(2^j n))$ being the $k$-th control point of $\mathcal{P}^{(j)}$ and other points on the control polygon determined by linear interpolation.  Closed parametric subintervals are of the form $[(k/(2^j n ), ((k + 1)/(2^j n)].$  The presence of consecutive collinear control points has implications for the discrete derivative.

In Figure~\ref{fig:shortside}, the triangle depicts the subtraction of any two vectors, where upper case letters denote angles, while lower case letters denote edges.   Consider {\bf A} relative to an upper bound on $\|{\bf a} \|$, where the direction of the vector corresponding to edge $\|{\bf a} \|$ is unknown.  The dotted circle indicates the positions that this edge could assume.  By symmetry, only the semi-circle above the horizontal axis needs to be considered.  When $\| {\bf a}\| < \|{\bf b}\|$, angle {\bf A} attains its maximum when angle {\bf C} is $\pi/2$, by elementary trigonometry.
\begin{figure}[h]
\centering{
    \scalebox{1.5}{\includegraphics[height=3.0in]{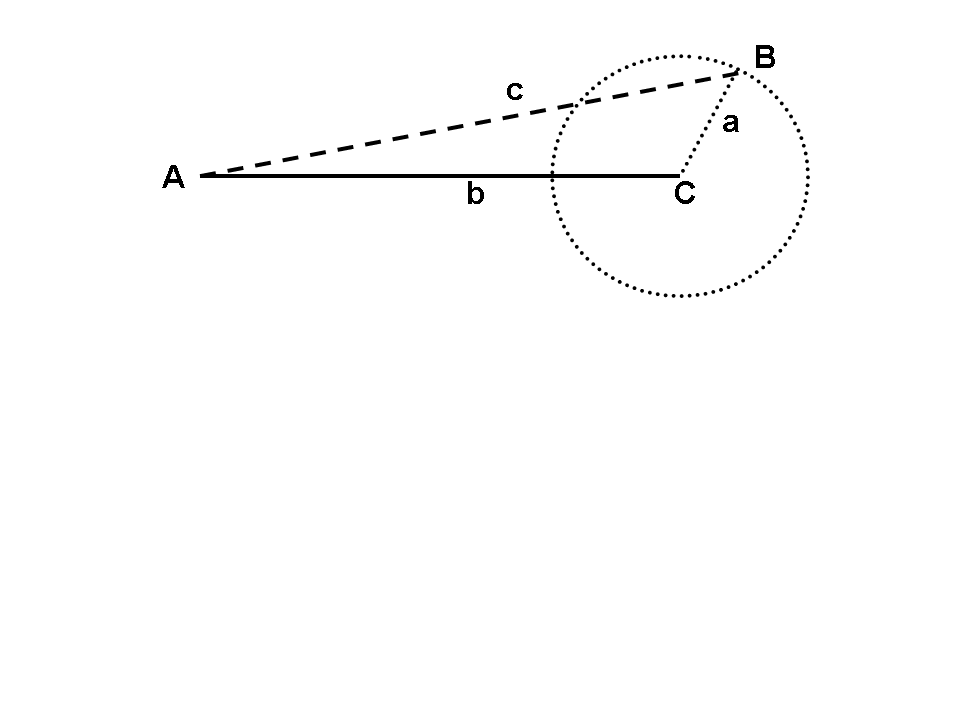}}}
\vspace*{-38ex}
\caption{Upper Bound on Angle}
 \label{fig:shortside}
\end{figure}
\begin{lemma}
\label{lem:maxangle}
The maximum angle between tangent vectors of $\B^{(j)}$ and  $\mathcal{P}$ 
 can be determined over the shortest edge of  $\mathcal{P}$.
\end{lemma}

\begin{proof}
For each $t \in [0,1]$, consider $\mathbf {\frac{d}{dt}}  \B^{(j)}(t) -  \mathcal{H}^{(j)}(t)$,  with edge {\bf b} in Figure~\ref{fig:shortside} depicting the tangent of ${\mathcal P}$ that corresponds to $\mathbf {\frac{d}{dt}}  \B^{(j)}(t).$

If $t$ is in an interval $[(k/(2^j n ), ((k + 1)/(2^j n)]$ which does not contain some $i/n$ for any $i = 0, \ldots, n$, then {\bf b} $=   n \|P_{i+1} - P_i\|$, for some $i = 0, \ldots, n - 1.$ 
 In the notation of Figure~\ref{fig:shortside}, the maximum value for angle {\bf A} will be given by arcsin(${\mathbf a}/{\mathbf b}$).

Denote by $\lambda$ the minimum length over all edges from $\mathcal{P}$ and let {\bf b} be a tangent vector of the form $n \ast (P_ i - P_{i - 1})$, for some $i = 1, \ldots, n.$
The condition on the length of {\bf a} is met by choosing $j$ such that the right hand side of Inequality~\ref{ineq:hod-rate} is less than $n\lambda$.  \hspace*{5ex} $\Box$
\end{proof}


\section{An Example}
\label{sec:example}

Following from Lemma~\ref{lem:maxangle} and Inequality~\ref{ineq:hod-rate} the values to satisfy Theorem~\ref{thm:delta-theta-iso} can be determined with an integer $m_1$ sufficient to achieve the distance bound of $\delta$ and with an integer $m_2$ sufficient to achieve the angular bound of $\pi/8$.  Let $M = max \{m_1, m_2\}$, so that for all $j > M$, $\B^{(j)}$ is isotopic to $\mathcal{P}.$

A detailed example follows on the control polygon \cite{LiPeMaKoJo15}
\vspace*{-1ex}
{\small $$(1.3076,   -3.3320,   -2.5072), (-1.3841, 4.6826, 0.9135), (-3.2983, -4.0567, 2.6862),$$ 
$$(-0.1233, 2.7683, -2.4636), (3.9080, -4.5334, 1.2264), (-3.9360, -0.4383, -0.9834),$$
$$(3.2182, 4.2961, 2.1125).$$}


\vspace*{-2ex}
\subsection{Empirically Determined Bounds}
\label{ssec:empstud}

Figure~\ref{fig:ta0} depicts the stick knot $4_1$, successively followed by the B\'ezier curves of the unknot from the $0$th collinear insertion and of $4_1$ from the $4$th collinear insertion, as previously proven \cite{LiPeMaKoJo15}. Figure~\ref{fig:itpro3} shows Iterations 1 - 3 of collinear insertion under projections that provide visual verification that each of the corresponding B\'ezier curves is the unknot.
\begin{figure}[htb]
\centering
 \subfigure[Stick Knot $4_1$]
    {
   \includegraphics[height=3cm]{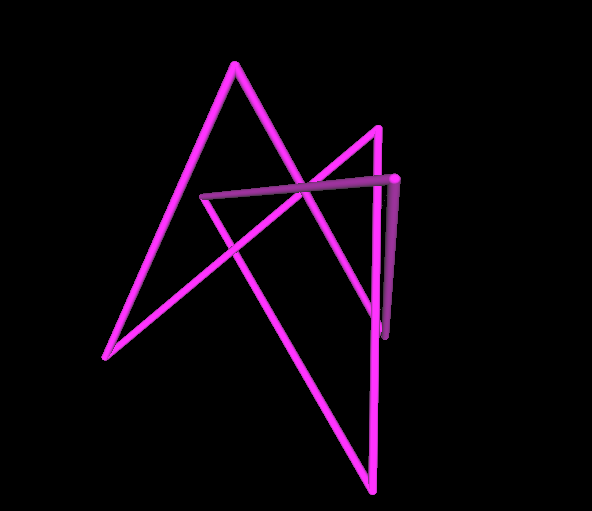}
   \label{fig:stick}
    }
 \subfigure[0th Iteration, $0_1$]
    {
   \includegraphics[height=3cm]{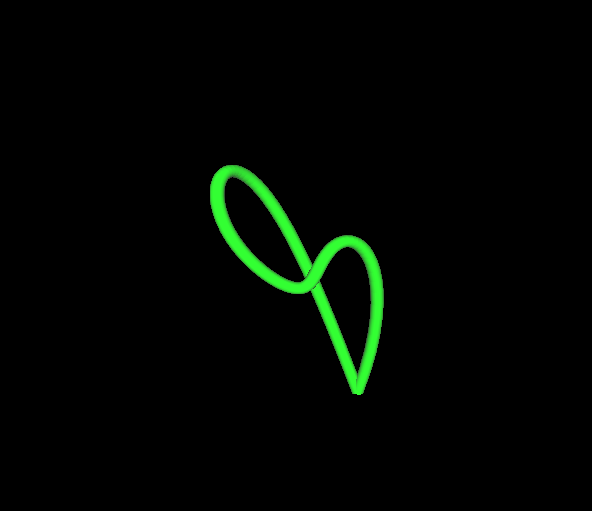}
    \label{fig:onlyunknot}
    }
 \subfigure[4th Iteration, $4_1$]
    {
   \includegraphics[height=3cm]{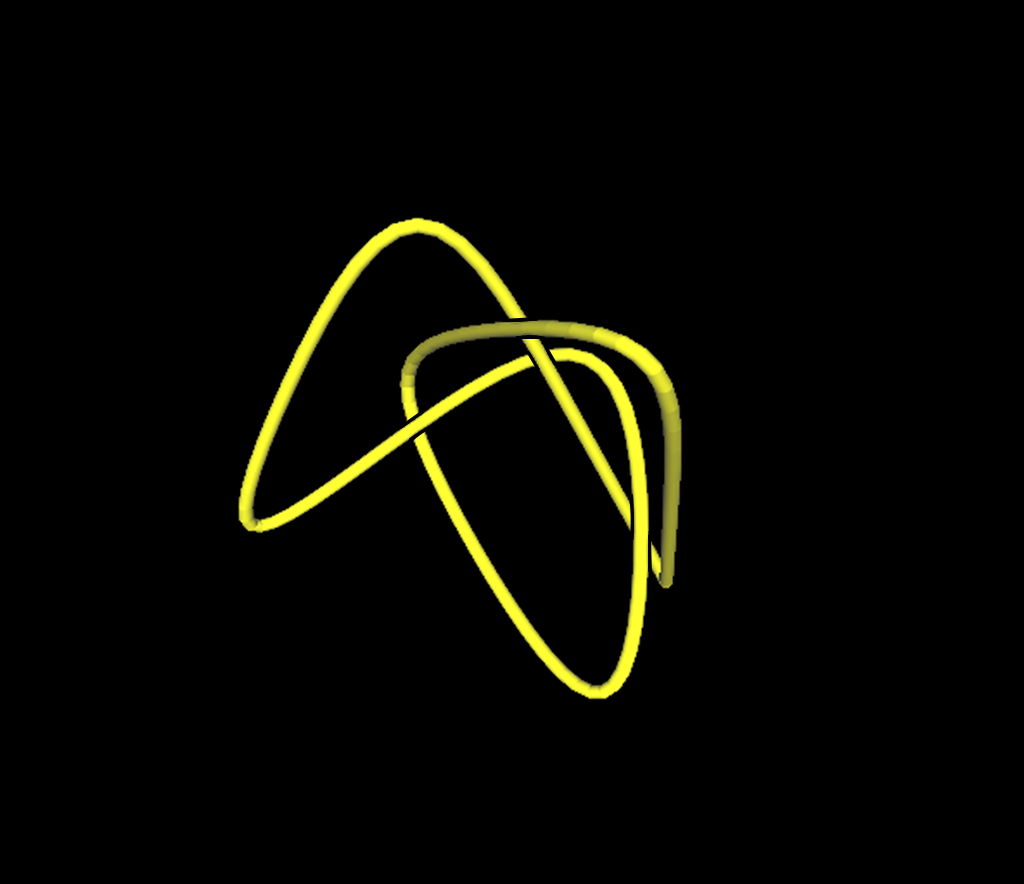}
   \label{fig:isok}
    }
    \vspace{-1ex}
\caption{Stick Knot, 0th \& 4th Collinear Insertions.} 
\label{fig:ta0}
\end{figure}
\vspace{-1ex}

\begin{figure}[htb]
\centering
 \subfigure[1st Iteration]
    {
   \includegraphics[height=3cm]{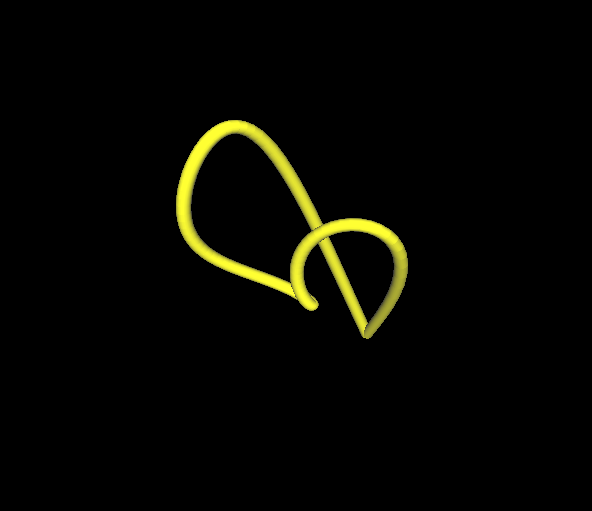}
   \label{fig:pro1}
    }
 \subfigure[2nd Iteration]
    {
   \includegraphics[height=3cm]{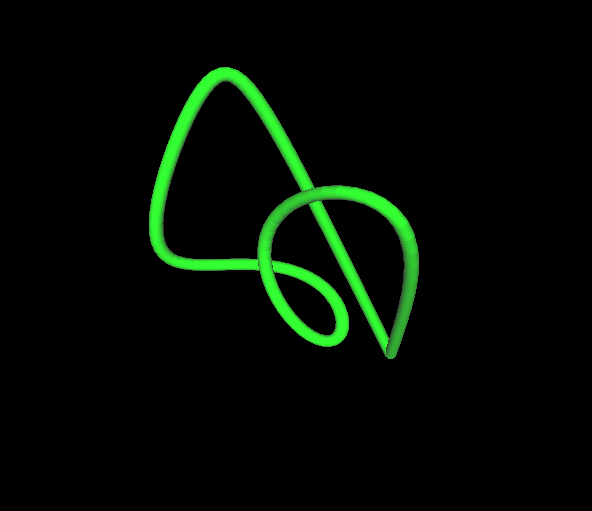}
    \label{fig:pro2}
    }
 \subfigure[3rd Iteration]
    {
   \includegraphics[height=3cm]{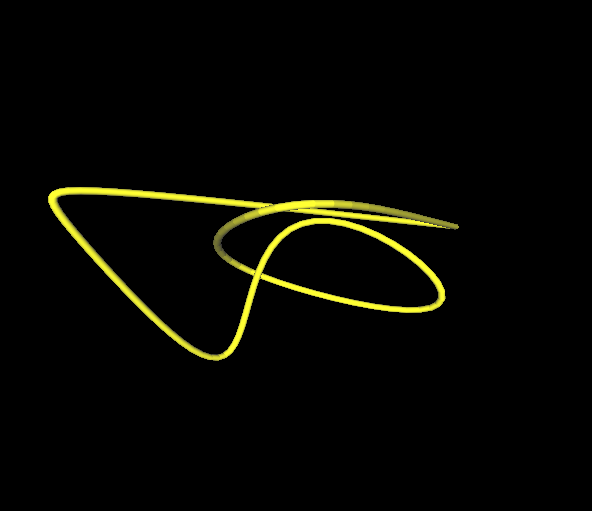}
   \label{fig:pro3}
    }
    \vspace{-1ex}
\caption{Projections of Three Iterations of Collinear Insertion.} 
\label{fig:itpro3}
\end{figure}
\vspace{-1ex}

\subsection{Theoretically Determined Bounds}
\label{ssec:theobounds}

To apply Theorem~\ref{thm:delta-theta-iso} note that $\mathcal{P}$ has finite total curvature.

The following narrative to compute $\delta$ is adapted from the source \cite{DenneSullivan2008} for the convenience of the reader.  Select a finite number of points $p_j$ from $\mathcal{P}$ (including all its control points) such that these points divide $\mathcal{P}$ into PL arcs $\alpha_k$, each of total curvature less than $\pi/8$. 
Let $\epsilon = 1.0$

Let $r_1$ be the minimum distance between any two arcs $a_k$ which are not incident
to a common $p_j$ (or the minimum distance between points $p_j$, if this is smaller). 
Let \[r_2 = \textnormal{min}(r_1/2, \epsilon/2).\]

Consider disjoint open balls $B_j$ of radius $r_2$ centered at the $p_j$ .  Note that
$\mathcal{P} \setminus \bigcup B_j$ is a compact union of disjoint arcs $\beta_k \subset \alpha_k$. Let $r_3$ be the minimum distance between any two of these arcs $\beta_k$

Let $r_4 = r_3/6$ and let $\delta = r_4/3.$

The values for $\mathcal{P}$ are summarized in Table~\ref{table:cdelta}.

\begin{table}[ht]
\begin{center}
\begin{tabular} {|l|r|}
\hline 
\multicolumn{1}{|l|} {\bf Variable} &
\multicolumn{1}{|r|} {\bf Value} \\ 
\hline
$\epsilon$ & $1.0000$  \\
\hline
$r_1$ & $0.2576$  \\
\hline
$r_2$ & $0.1288$  \\
\hline
$r_3$ & $0.0576$ \\  
\hline
$r_4$ & $0.0096$ \\ 
\hline
$\delta$ & $0.0032$ \\ 
\hline
\end{tabular}
\end{center}
\caption{Computing $\delta$.}
\label{table:cdelta}
\end{table}

The values for $m_1$ and $m_2$ discussed at the beginning of this section can now be calculated, using $n = 7, \lambda = 9$.

\pagebreak

To obtain $m_1$, use Inequality~\ref{ineq:drate} and let $\delta = (n/(4 \sqrt{n \ast 2^j + 1})) \ast \| \Delta_{2}P \|_{1,M}$.   Let $ \Omega_1 = \| \Delta_{2}P \|_{1,M} = 70.8$, so that
\[ m_1  = \lceil \textnormal{log}_2 ( (7 \Omega_1^2)/(16 \delta^2) - (1/7) ) \rceil = 28. \]

There are two constraints on $m_2$ in Lemma~\ref{lem:maxangle}, with the first being relative to the shortest tangent vector and the second in terms of the angular bound of $\pi/8$.  The corresponding number of iterations to meet those constraints will be denoted by $m_{2T}$ and $m_{2A}$, respectively, with $m_2 = \max\{m_{2T}, m_{2A}\}.$

First, $m_{2T}$ is computed following Lemma~\ref{lem:maxangle} to have $\| a\| < n \lambda$, with
$\|a \|$ set equal to the right hand side (RHS) of Inequality\ref{ineq:hod-rate} and $\Omega_2 = \| \Delta_{2}H \|_{1,M} = 779.0$ to get
\[ (n/ (2 \ast \sqrt{n \ast 2^j)} \ast 2^{-(j - 1)} \ast \Omega_2 < n  \lambda, \]
and
\[ (\Omega_2/\lambda)^2 < n 2^{3j} + 2^{2j}, \]
to yield $m_{2T} = 4.$

For $m_{2A},$ again, set $\|a\| $  equal to the RHS of Inequality\ref{ineq:hod-rate} and use $\arcsin(\|a\|/(n\lambda)) = \pi/8$ to get $m_{2A} = 5$.  It is then clear, that $m_2 = 5$ and $M = 28.$

The difference between the theoretically determined sufficient condition of 28 iterations to achieve isotopic equivalence versus the visual inspection that 4 iterations suffice present opportunities to explore tightness criteria.  These future investigations will be enabled by previously developed knot visualization software \cite{TJPweb}.  It is of interest to note that the empirically determined optimal value of 4 iterations for the example is exactly met by the tangency condition, with $m_{2T} = 4$, and is nearly met by the angular condition, with $m_{2A} = 5$.  This directs initial attention to considering why the distance criteria yield such a larger number of 28, while also appreciating that the tangency and angular conditions for this example might just be particularly well-behaved.

\section{Conclusion and Future Work} 
For synchronous visualizations of writhing molecules, a cautionary example was previously presented of different knot types between a polynomial curve and its rendering.  Further study of the isotopy constraints on that example prompted these generalizations for a sequence B\'ezier curves converging  to a given stick knot.  This broader theory will further inform the design of visualization software for molecular simulations.  Beyond that application, this work lays the foundation for mathematical visualization software for the experimental investigation of important theoretical relationships between stick and smooth knots, where much remains to be discovered.  A particular future emphasis will be upon whether isotopic approximations can occur earlier in the converging sequence than given by the sufficient \emph{a prior} bounds presented here.  Any discovery of more aggressive bounds will likely rely upon computational experiments to discover criteria for tight bounds.

\pagebreak

\section{Appendix: Supportive Combinatorial Relations}
\label{sec:suppcombo}

Let  $\N$ denote the natural numbers $\{1, 2, 3, \ldots \}$.  For $k \in \N$, consider the central binomial coefficient \cite{matho}, 
\[ \left( \begin{array}{c} 2k \\ k \end{array} \right).\] 

\begin{lemma} For $k \in \N$,
\[ \left( \begin{array}{c} 2k \\ k \end{array} \right) \leq \frac{4^k}{\sqrt{2k + 1}}. \]
\label{lem:erd}
\end{lemma}

\begin{proof} The key to the proof has informally been attributed to P. Erdos \cite{matho} . \hspace*{5ex} $\Box$
\end{proof}

\vspace{2ex}

The following notation has previously appeared   \cite{Nairn-Peters-Lutterkort1999} and is central to the convergence results.   Note that presenting the definition purely for even arguments is sufficient here, since the colliear insertion process always doubles the degree of the previous B\'ezier curve.  For $k \in N$, let 
\[N_1(2k) = \left( \begin{array}{c} 2k \\ k \end{array} \right) \frac{2k}{2^{{2k} + 2}}.\] 

\begin{lemma}
For $k \in N$, 
\[N_1(2k) < \frac{k}{2\sqrt{2k + 1}}.\]
\label{lem:N1ubnd}
\end{lemma}

\begin{proof}  
The short proof \cite{LiPeMaKoJo15} is repeated here for the convenience of the reader.
Invoke Lemma~\ref{lem:erd} on
\[ N_1(2k) = \left( \begin{array}{c} 2k \\ k \end{array} \right) \frac{2k}{2^{{2k} + 2}} < 
\frac{4^k}{\sqrt{2k + 1}} \frac{2k}{4^{k + 1}} = \frac{k}{2 \sqrt{2k + 1}}.  \hspace*{5ex} \Box  \] 
\end{proof}

\pagebreak


\bibliographystyle{elsarticle-harv} 
\bibliography{ji-tjp-biblio2}





\end{document}